\documentclass[11pt,a4paper]{amsart}

\usepackage[utf8]{inputenc}
\usepackage[T1]{fontenc}
\usepackage{amsmath,amssymb,amsthm,amsfonts}
\usepackage{mathtools} 
\usepackage{tikz-cd}   
\usepackage[colorlinks=true,linkcolor=blue,citecolor=red]{hyperref}
\usepackage{geometry}
\newtheorem{theorem}{Theorem}[section]
\newtheorem{lemma}[theorem]{Lemma}

\theoremstyle{definition}
\newtheorem{definition}[theorem]{Definition}
\newtheorem{remark}[theorem]{Remark}

\numberwithin{equation}{section}

\newcommand{\Ind}{\mathrm{Ind}}
\newcommand{\Res}{\mathrm{Res}}
\newcommand{\bInd}{\mathbf{Ind}}
\newcommand{\bRes}{\mathbf{Res}}

\newcommand{\Ql}{\overline{\mathbb{Q}}_\ell}

\title{The Iterated Geometric Green's Formula}
\author{Chao Shen}
\thanks{This work was partially supported by the National Natural Science Foundation of China (No. 12471030)}
\address{Department of Mathematical Sciences, Tsinghua University, Beijing, China}
\email{sc20@mails.tsinghua.edu.cn}

\author[J. Xiao]{Jie Xiao}
\address{School of Mathematical Sciences, Beijing Normal University, Beijing, China}
\email{jxiao@bnu.edu.cn}
\date{\today}

\begin{document}

\begin{abstract}
	Fang, Lan, and Xiao established the geometric Green's formula as a categorical isomorphism for arbitrary semisimple complexes. In this short note, we generalize their work to multi-step compositions. Specifically, we establish the iterated geometric Green's formulas for the composition of an $(n-1)$-fold restriction and an induction, as well as its dual.
\end{abstract}

	\maketitle
	
\section{Introduction}

In the geometric categorification of quantum groups \cite{Lus90, Lus91, Lus93}, Lusztig introduced the induction and restriction functors on quiver representation spaces to categorify the multiplication and comultiplication of quantum groups. He proved a fundamental formula relating the composition of restriction and induction functors for a specific subcategory of perverse sheaves arising from flag varieties.

At the algebraic level, Green \cite{Green95} established a homological formula for filtration numbers, which equips the Hall algebra with a bialgebra structure. For complexes beyond Lusztig's specific subcategory, Xiao, Xu, and Zhao \cite{XXZ} verified the corresponding formula by utilizing trace maps and Green's algebraic formula under the sheaf-function correspondence. Subsequently, Fang, Lan, and Xiao \cite{FLX} categorified this relation, establishing the geometric Green's formula as an isomorphism for arbitrary semisimple complexes. In our attempt to formulate quantum cluster $\mathbf{i}$-characters \cite{BerensteinRupel2015} within Lusztig's geometric framework, an iterated version of Green's formula naturally emerges. We believe this iterated relation is of independent importance.

In this short note, we establish iterated geometric Green's formulas for the multi-step compositions of restriction and induction functors. For instance, the composition of an $(n-1)$-fold iterated restriction and an induction evaluated on semisimple complexes expands as:
	\[
	\bRes_{\gamma_1, \dots, \gamma_n} \circ \mathrm{Ind}_{\alpha, \beta} \cong \bigoplus_{\lambda} (\Psi_{\text{final}} \circ \cdots \circ \Psi_1) \{ -\Xi_{\lambda} \}
	\]
	where each operator $\Psi_j$ is constructed from single inductions and restrictions. A dual expansion holds for the composition $\mathrm{Res}_{\alpha, \beta} \circ \bInd_{\gamma_1, \dots, \gamma_n}$.
	
The paper is organized as follows. Section 2 recalls the constructions of the induction and restriction functors. Section 3 states the iterated geometric Green's formulas, and Section 4 provides their detailed proofs.
	
\section{Preliminaries}
Let $q$ be a prime power, $\mathbb{F}_q$ be a finite field, and $\overline{\mathbb{F}}_q$ be its algebraic closure. Throughout this paper, we work within the bounded equivariant derived category of constructible sheaves $\mathcal{D}^b_{G}(-, \overline{\mathbb{Q}}_\ell)$, and we restrict our attention to the full subcategory of semisimple complexes.

\subsection{Representation Spaces of Quivers}
Let $Q = (Q_0, Q_1)$ be a finite quiver, where $Q_0$ is the set of vertices and $Q_1$ is the set of arrows.

\begin{definition}[Euler Form]\label{def: euler-form}
	For any dimension vectors $\alpha, \beta \in \mathbb{N}^{Q_0}$, the \textbf{Euler form} is defined as the bilinear form:
	\begin{equation}
		\langle \alpha, \beta \rangle = \sum_{i \in Q_0} \alpha_i \beta_i - \sum_{h \in Q_1} \alpha_{s(h)} \beta_{t(h)}.
	\end{equation}
	The \textbf{symmetrized Euler form} is given by $(\alpha, \beta) = \langle \alpha, \beta \rangle + \langle \beta, \alpha \rangle$.
\end{definition}

\begin{definition}[Representation Space]\label{def: rep-space}
	For a dimension vector $\alpha \in \mathbb{N}^{Q_0}$, the \textbf{representation space} is defined as the affine variety:
	\begin{equation}
		E_{\alpha} = \bigoplus_{h\in Q_1} \mathrm{Hom}(\overline{\mathbb{F}}_q^{\alpha_{s(h)}}, \overline{\mathbb{F}}_q^{\alpha_{t(h)}}).
	\end{equation}
	The \textbf{gauge group} $G_{\alpha} = \prod_{i\in Q_0} \mathrm{GL}_{\alpha_i}(\overline{\mathbb{F}}_q)$ acts on $E_{\alpha}$ naturally by conjugation:
	\begin{equation}
		(g \cdot x)_h = g_{t(h)} x_h g_{s(h)}^{-1}
	\end{equation}
	for $g\in G_{\alpha}$, $x\in E_{\alpha}$, and $h\in Q_1$.
\end{definition}

\subsection{Induction and Restriction Functors}

In Lusztig's geometric categorification of quantum groups \cite{Lus90,Lus91,Lus93}, the induction and restriction functors correspond to the algebraic multiplication and comultiplication, respectively. We now briefly recall their definitions.

\textbf{Induction.} Let $\alpha, \beta \in \mathbb{N}^{Q_0}$ be dimension vectors. We define the variety $E_{\alpha,\beta}$ as the space of pairs $(x, W)$, where $x \in E_{\alpha+\beta}$ and $W \subset \overline{\mathbb{F}}_q^{\alpha+\beta}$ is an $x$-stable $Q_0$-graded subspace of dimension vector $\beta$. Let $E^{(1)}_{\alpha, \beta}$ be the variety of quadruples $(x, W, \rho_{\alpha}, \rho_{\beta})$, where $(x, W) \in E_{\alpha, \beta}$ and $\rho_{\alpha}: \overline{\mathbb{F}}_q^{\alpha+\beta}/ W \xrightarrow{\sim} \overline{\mathbb{F}}_q^{\alpha}$ and $\rho_{\beta}: W \xrightarrow{\sim} \overline{\mathbb{F}}_q^{\beta}$ are linear isomorphisms.

Consider the following induction diagram:
\begin{equation}
	\begin{tikzpicture}[baseline=(current bounding box.center), >=stealth, scale=1.2]
		\node (L) at (0, 0) {$E_\alpha \times E_\beta$};
		\node (M1) at (3, 0) {$E^{(1)}_{\alpha, \beta}$};
		\node (M2) at (6, 0) {$E_{\alpha,\beta}$};
		\node (R) at (9, 0) {$E_{\alpha+\beta}$};
		
		\draw[->] (M1) -- node[above] {$p_{\alpha,\beta}$} (L);
		\draw[->] (M1) -- node[above] {$r_{\alpha,\beta}$} (M2);
		\draw[->] (M2) -- node[above] {$q_{\alpha,\beta}$} (R);
	\end{tikzpicture}
\end{equation}
Here, $p_{\alpha,\beta}(x, W, \rho_{\alpha}, \rho_{\beta}) = (\rho_{\alpha,*}(x|_{\mathbb{F}^{\alpha+\beta}_q/ W}), \rho_{\beta,*}(x|_W))$ is a smooth morphism with connected fibers. The maps $r_{\alpha,\beta}$ and $q_{\alpha,\beta}$ are natural projections, where $r_{\alpha,\beta}$ is a principal $G_\alpha \times G_\beta$-bundle, and $q_{\alpha,\beta}$ is proper.

\begin{definition}[Induction]\label{def: ind}
	For $\mathcal{F} \in \mathcal{D}_{G_\alpha \times G_\beta}^b (E_\alpha \times E_\beta, \overline{\mathbb{Q}}_\ell)$, the \textbf{induction functor} is defined as:
	\begin{equation}
		\mathrm{Ind}_{\alpha,\beta}(\mathcal{F}) = (q_{\alpha,\beta})_! (r_{\alpha,\beta})_\flat (p_{\alpha,\beta})^*(\mathcal{F}) \{ \langle\langle \alpha, \beta \rangle\rangle \},
	\end{equation}
where $(r_{\alpha,\beta})_\flat$ is the quasi-inverse of the pullback $(r_{\alpha,\beta})^*$, and $\langle\langle \alpha, \beta \rangle\rangle = \sum_{i \in Q_0} \alpha_i \beta_i + \sum_{h \in Q_1} \alpha_{s(h)} \beta_{t(h)}$. Here, the notation $\{ d \} = [d](d/2)$ denotes the shift and Tate twist.
\end{definition}

\textbf{Restriction.} Fix a $Q_0$-graded subspace $W_0 \subset \overline{\mathbb{F}}_q^{\alpha+\beta}$ of dimension vector $\beta$, with chosen isomorphisms $\rho_{\alpha,0}: \overline{\mathbb{F}}_q^{\alpha+\beta} / W_0 \xrightarrow{\sim} \overline{\mathbb{F}}_q^\alpha$ and $\rho_{\beta,0}: W_0 \xrightarrow{\sim} \overline{\mathbb{F}}_q^\beta$. Let $F_{\alpha,\beta} \subset E_{\alpha+\beta}$ be the closed subvariety of representations stabilizing $W_0$, and let $P_{\alpha,\beta} \subset G_{\alpha+\beta}$ be the corresponding parabolic subgroup.

Consider the restriction diagram:
\begin{equation}
	\begin{tikzpicture}[baseline=(current bounding box.center), >=stealth, scale=1.2]
		\node (L) at (0, 0) {$E_\alpha \times E_\beta$};
		\node (M) at (3, 0) {$F_{\alpha,\beta}$};
		\node (R) at (6, 0) {$E_{\alpha+\beta}$};
		
		\draw[->] (M) -- node[above] {$\kappa_{\alpha,\beta}$} (L);
		\draw[->] (M) -- node[above] {$\iota_{\alpha,\beta}$} (R);
	\end{tikzpicture}
\end{equation}
where $\iota_{\alpha,\beta}$ is the natural embedding, and $\kappa_{\alpha,\beta}(x) = ((\rho_{\alpha,0})_*(x|_{\overline{\mathbb{F}}_q^{\alpha+\beta}/ W_0}), (\rho_{\beta,0})_*(x|_{W_0}))$ is a vector bundle of rank $\sum_{h\in Q_1}\alpha_{s(h)}\beta_{t(h)}$, equivariant under the parabolic action.

\begin{definition}[Restriction]\label{def: res}
	For $\mathcal{L} \in \mathcal{D}_{G_{\alpha+\beta}}^b(E_{\alpha+\beta}, \overline{\mathbb{Q}}_\ell)$, the \textbf{restriction functor} is defined as:
	\begin{equation}
		\mathrm{Res}_{\alpha,\beta}(\mathcal{L}) = (\kappa_{\alpha,\beta})_! (\iota_{\alpha,\beta})^* (\mathcal{L}) \{ -\langle \alpha, \beta \rangle \}.
	\end{equation}
\end{definition}

\begin{lemma}[\cite{BBD, Braden03}; see also \cite{Lus93}]\label{lem: semisimplicity}
	Both the induction functor $\mathrm{Ind}_{\alpha,\beta}$ and the restriction functor $\mathrm{Res}_{\alpha,\beta}$ preserve the subcategory of semisimple complexes.
\end{lemma}

\subsection{Geometric Green's Formula}

Let $Q = (Q_0, Q_1)$ be a quiver. For a sequence of dimension vectors $\underline{\alpha} = (\alpha_1, \dots, \alpha_m)$ with $\sum_{j=1}^m \alpha_j = \alpha$, let $E_{\alpha_1, \dots, \alpha_m}$ be the variety of $\underline{\alpha}$-filtered $Q$-representations. The variety of filtrations $\tilde{{F}}_{\underline{\alpha}}$ is defined as the variety of pairs $(x, \mathbf{V})$ where $x \in E_\alpha$ and $\mathbf{V} = (0=V^0 \subset \dots \subset V^m = V_\alpha)$ is a filtration of subrepresentations satisfying $\underline{\dim}(V^j/V^{j-1}) = \alpha_j$. Let $\pi_{\underline{\alpha}}: \tilde{{F}}_{\underline{\alpha}} \to E_\alpha$ be the proper projection mapping $(x, \mathbf{V}) \mapsto x$.

\begin{definition}[Lusztig Sheaf]
	The \textbf{Lusztig sheaf} associated with $\underline{\alpha}$ is defined as the complex:
	$$\mathcal{L}_{\underline{\alpha}} \coloneqq (\pi_{\underline{\alpha}})_! \underline{\overline{\mathbb{Q}}_\ell}_{\tilde{{F}}_{\underline{\alpha}}} \{d_{\underline{\alpha}}\}$$
	where $d_{\underline{\alpha}} = \dim E_{\alpha_1, \dots, \alpha_m} + \sum_{i=1}^m \dim E_{\alpha_i}$. Equivalently, it can be expressed as the $m$-fold iterated induction product $\mathbf{1}_{\alpha_1} \star \dots \star \mathbf{1}_{\alpha_m}$ with $\mathbf{1}_{\alpha_i} \coloneqq \underline{\overline{\mathbb{Q}}_\ell}_{E_{\alpha_i}}\{\dim E_{\alpha_i}\}$.
\end{definition}

\begin{remark}
	Since $\pi_{\underline{\alpha}}$ is proper, $\mathcal{L}_{\underline{\alpha}}$ is a semisimple complex. In the case where each $\alpha_k$ is a multiple of a simple dimension vector, $\mathcal{L}_{\underline{\alpha}}$ matches the monomial product $u_{\alpha_1} \dots u_{\alpha_m}$ in the composition algebra of the Hall algebra under the sheaf-function correspondence. 
\end{remark}

\begin{lemma}[Restriction Formula \cite{Lus90,Lus91,Lus93, Schiffmann12}]\label{lem: Res-formula}
	Let $\underline{\alpha} = (\alpha_1, \dots, \alpha_m)$ be a sequence where each $\alpha_k$ is a multiple of a simple dimension vector $\epsilon_{i_k}$ (corresponding to a vertex $i_k \in Q_0$). The restriction of $\mathcal{L}_{\underline{\alpha}}$ admits a decomposition:

	\begin{equation}
		\mathrm{Res}_{\beta, \gamma}^\alpha(\mathcal{L}_{\underline{\alpha}}) = \bigoplus_{(\underline{\beta}, \underline{\gamma})} \mathcal{L}_{\underline{\beta}} \boxtimes \mathcal{L}_{\underline{\gamma}} \{\Delta_{\underline{\beta}, \underline{\gamma}} \},
	\end{equation}
	where the direct sum runs over all sequences $\underline{\beta}, \underline{\gamma} \in (\mathbb{N}^{Q_0})^m$ satisfying:
		\begin{equation}
			\sum_{k=1}^m \beta_k = \beta, \quad \sum_{k=1}^m \gamma_k = \gamma, \quad \text{and } \beta_k + \gamma_{k} = \alpha_k \text{ for all } k.
		\end{equation}
	The normalized shift $\Delta_{\underline{\beta}, \underline{\gamma}}$ is given by:
	\begin{equation}
		\Delta_{\underline{\beta}, \underline{\gamma}} = \dim E_{\underline{\alpha}} - \dim E_{\underline{\beta}} - \dim E_{\underline{\gamma}} - \langle \beta, \gamma \rangle - 2\operatorname{rank} \kappa^{\underline{\beta}, \underline{\gamma}},
	\end{equation}
	with
\begin{equation}
	\operatorname{rank} \kappa^{\underline{\beta}, \underline{\gamma}}= 
	\sum_{i\in Q_0} \sum_{k<l} (\beta_l)_i(\gamma_k)_i +
	\sum_{h\in Q_1} \sum_{k<l} (\beta_k)_{s(h)} (\gamma_l)_{t(h)}.
\end{equation}
\end{lemma}

\begin{theorem}[Geometric Green's Formula \cite{FLX}]\label{thm: Lusztig-Green}
	 Let $\alpha, \beta, \alpha', \beta' \in \mathbb{N}^{Q_0}$ satisfy $\alpha + \beta = \alpha' + \beta' = \gamma$. For any semisimple complexes $\mathcal{F} \in \mathcal{D}^b_{G_\alpha}(E_\alpha)$ and $\mathcal{G} \in \mathcal{D}^b_{G_\beta}(E_\beta)$, there is an isomorphism:
	 \begin{equation}
	 	\begin{aligned}
	 		\mathrm{Res}_{\alpha', \beta'}^{\gamma} \left( \mathrm{Ind}_{\alpha,\beta}^{\gamma} (\mathcal{F} \boxtimes \mathcal{G}) \right) \simeq &
	 		\bigoplus_{\lambda \in \mathcal{N}} 
	 		\left( \mathrm{Ind}_{\alpha_a, \beta_a}^{\alpha'} \times \mathrm{Ind}_{\alpha_b, \beta_b}^{\beta'} \right) \circ \tau_{\lambda!} \\
	 		& \circ \left( \mathrm{Res}_{\alpha_a, \alpha_b}^{\alpha} \mathcal{F} \boxtimes \mathrm{Res}_{\beta_a, \beta_b}^{\beta} \mathcal{G} \right) 
	 		\{ - ( \alpha_b, \beta_a ) \},
	 	\end{aligned}
	 \end{equation}
	 	where the direct sum runs over the set $\mathcal{N}$ of all quadruples $\lambda = (\alpha_a, \alpha_b, \beta_a, \beta_b) \in (\mathbb{N}^{Q_0})^4$ satisfying:
	 	\begin{equation}
	 		\alpha = \alpha_a + \alpha_b, \quad \beta = \beta_a + \beta_b, \quad \alpha' = \alpha_a + \beta_a, \quad \beta' = \alpha_b + \beta_b.
	 	\end{equation}
     $\tau_{\lambda}: E_{\alpha_a} \times E_{\alpha_b} \times E_{\beta_a} \times E_{\beta_b} \xrightarrow{\simeq} E_{\alpha_a} \times E_{\beta_a} \times E_{\alpha_b} \times E_{\beta_b}$ is the coordinate exchange isomorphism swapping the second and third factors, and $\{ - ( \alpha_b, \beta_a ) \}$ is the normalized shift given by the negative symmetrized Euler form.
\end{theorem}

\begin{remark}
	The geometric Green's formula was originally established by Lusztig in \cite{Lus90, Lus91, Lus93} for the subcategory $\mathcal{Q}$ additively generated by the direct summands of Lusztig sheaves, and was later extended to all semisimple complexes in \cite{XXZ,FLX}. 
	Although a general semisimple complex may not be explicitly decomposed into a direct sum of external tensor products under the restriction functor, this extension remains fundamental.
\end{remark}

\section{Main Results: Iterated Geometric Green's Formulas}

In this section, we present the iterated generalization of the geometric Green's formula for semisimple complexes.

\subsection{Iterated Functors}
Fix an integer $n \geq 2$. For a dimension vector sequence $(\gamma_1, \dots, \gamma_n)$, we denote the \textbf{tail sums} by $\gamma_{\ge j} = \sum_{k=j}^n \gamma_k$. In particular, for $j=n$, the sum reduces to $\gamma_{\ge n} = \gamma_n$. We also denote the total sum by $\gamma \coloneqq \gamma_{\ge 1}$.

\begin{definition}[Iterated Induction and Restriction Functors]
	Fix an integer $n \geq 2$. Let $\underline{\gamma} = (\gamma_1, \dots, \gamma_n)$ be a sequence of dimension vectors with total dimension $\gamma \coloneqq \gamma_{\ge 1}$.
	
	\begin{enumerate}
		\item The \textbf{$(n-1)$-fold iterated induction functor} 
		\begin{equation}
			\bInd_{\gamma_1, \dots, \gamma_n}^{\gamma} : D^b_{m, \prod_{i=1}^{n}G_{\gamma_{i}}}\left(\prod_{i=1}^{n} E_{\gamma_i}, \overline{\mathbb{Q}}_\ell\right) \to D^b_{m, G_\gamma}(E_{\gamma},\overline{\mathbb{Q}}_\ell).
		\end{equation}
		is defined via the recursive composition:
		\begin{equation}
			\bInd_{\gamma_1, \dots, \gamma_n}^{\gamma} \coloneqq \Ind_{\gamma_1, \gamma_{\ge 2}}^{\gamma} \circ \left(\mathrm{Id}_{\gamma_1} \times \Ind_{\gamma_2, \gamma_{\ge 3}}^{\gamma_{\ge 2}}\right) \circ \dots \circ \left(\mathrm{Id}_{\gamma_1 \times \dots \times \gamma_{n-2}} \times \Ind_{\gamma_{n-1}, \gamma_n}^{\gamma_{\ge n-1}} \right).
		\end{equation}
		
		\item Correspondingly, the \textbf{$(n-1)$-fold iterated restriction functor}
			\begin{equation}
			\bRes_{\gamma_1, \dots, \gamma_n}^{\gamma} : D^b_{m, G_\gamma}(E_{\gamma},\overline{\mathbb{Q}}_\ell) \to D^b_{m, \prod_{i=1}^{n}G_{\gamma_{i}}}\left(\prod_{i=1}^{n} E_{\gamma_i}, \overline{\mathbb{Q}}_\ell\right).
		\end{equation}
		is defined via the recursive composition:
		\begin{equation}
			\bRes_{\gamma_1, \dots, \gamma_n}^{\gamma} \coloneqq \left(\mathrm{Id}_{\gamma_1 \times \dots \times \gamma_{n-2}} \times \Res_{\gamma_{n-1}, \gamma_{n}}^{\gamma_{\ge n-1}} \right) \circ \dots \circ \left(\mathrm{Id}_{\gamma_1} \times \Res_{\gamma_2, \gamma_{\ge 3}}^{\gamma_{\ge 2}}\right) \circ \Res_{ \gamma_1, \gamma_{\ge 2}}^{\gamma}.
		\end{equation}
	\end{enumerate}
	Due to the associativity of the induction functor and the coassociativity of the restriction functor, these functors are well-defined up to canonical isomorphisms, making them independent of the order of parenthesization.
\end{definition}

\subsection{The formula for $\bRes \circ \Ind$}
The first result characterizes the composition of an $(n-1)$-fold iterated restriction functor with a binary induction functor. To state this main identity clearly, we first introduce the necessary index sets and operators.

For a given pair of dimension vectors $\alpha$ and $\beta$ with $\gamma = \alpha + \beta$, and an integer $n \ge 2$, we define the index set $\mathcal{M}_{n-1}$ as the collection of sequences of quadruples of dimension vectors:
\begin{equation}
	\lambda = (\lambda_1, \lambda_2, \dots, \lambda_{n-1}) \in \prod_{j=1}^{n-1} (\mathbb{N}^{Q_0})^4.
\end{equation}
Each component $\lambda_j = ( (\alpha_j)_a, (\alpha_j)_b, (\beta_j)_a, (\beta_j)_b )$ consists of the intermediate dimension vectors at the $j$-th step, which satisfy the following system of recursive constraints for $1 \le j \le n-1$:
\begin{equation}
	\begin{aligned}
		\gamma_j &= (\alpha_j)_a + (\beta_j)_a, & \gamma_{\ge j+1} &= (\alpha_j)_b + (\beta_j)_b, \\
		\alpha_j &= (\alpha_j)_a + (\alpha_j)_b, & \beta_j &= (\beta_j)_a + (\beta_j)_b,
	\end{aligned}
\end{equation}
with the initial conditions $\alpha_1 = \alpha$, $\beta_1 = \beta$, and $\alpha_j = (\alpha_{j-1})_b$, $\beta_j = (\beta_{j-1})_b$ for $2 \le j \le n-1$.

Associated with each sequence $\lambda \in \mathcal{M}_{n-1}$, we define the total shift as:
\begin{equation}
	\Xi = \sum_{j=1}^{n-1} \left( (\alpha_j)_b, (\beta_j)_a \right).
\end{equation}

Furthermore, we define a sequence of recursive operators $\Psi_j$ for $1 \le j \le n-2$ by:
\begin{equation}
	\begin{aligned}
		\Psi_j \coloneqq & \left( \mathrm{Id}_{\gamma_1, \dots, \gamma_{j-1}} \times \Ind_{(\alpha_j)_a, (\beta_j)_a}^{\gamma_j} \times \mathrm{Id}_{(\alpha_j)_b, (\beta_j)_b} \right) \circ \left( \mathrm{Id}_{\gamma_1, \dots, \gamma_{j-1}} \times \tau_{\lambda_j!} \right) \\
		& \circ \left( \mathrm{Id}_{\gamma_1, \dots, \gamma_{j-1}} \times \Res_{(\alpha_j)_a, (\alpha_j)_b}^{\alpha_j} \times \Res_{(\beta_j)_a, (\beta_j)_b}^{\beta_j} \right),
	\end{aligned}
\end{equation}
and the terminal operator $\Psi_{\text{final}}$ by:
\begin{equation}
	\begin{aligned}
		\Psi_{\text{final}} \coloneqq & \left( \mathrm{Id}_{\gamma_1, \dots, \gamma_{n-2}} \times \Ind_{(\alpha_{n-1})_a, (\beta_{n-1})_a}^{\gamma_{n-1}} \times \Ind_{(\alpha_{n-1})_b, (\beta_{n-1})_b}^{\gamma_n} \right) \circ \left( \mathrm{Id}_{\gamma_1, \dots, \gamma_{n-2}} \times \tau_{\lambda_{n-1}!} \right) \\
		& \circ \left( \mathrm{Id}_{\gamma_1, \dots, \gamma_{n-2}} \times \Res_{(\alpha_{n-1})_a, (\alpha_{n-1})_b}^{\alpha_{n-1}} \times \Res_{(\beta_{n-1})_a, (\beta_{n-1})_b}^{\beta_{n-1}} \right).
	\end{aligned}
\end{equation}

With these notations established, the main identity can be stated as follows.

\begin{theorem} \label{thm:main1}
	Let $\mathcal{F} \in \mathcal{D}^b_{G_{\alpha}}(E_{\alpha}, \Ql)$ and $\mathcal{G} \in \mathcal{D}^b_{G_{\beta}}(E_{\beta}, \Ql)$ be semisimple complexes, and let $\gamma = \alpha + \beta$. The composition of the iterated restriction and the induction functor satisfies the following isomorphism:
	\begin{equation}
		\bRes_{\gamma_1, \dots, \gamma_n}^{\gamma} \circ \Ind_{\alpha, \beta}^{\gamma}(\mathcal{F} \boxtimes \mathcal{G}) 
		\simeq \bigoplus_{\lambda \in \mathcal{M}_{n-1}} \left( \Psi_{\text{final}} \circ \Psi_{n-2} \circ \dots \circ \Psi_1 \right) (\mathcal{F} \boxtimes \mathcal{G}) \{ -\Xi \}.
	\end{equation}
\end{theorem}

\subsection{The formula for $\Res \circ \bInd$}

Analogously, we establish the decomposition for the composition of a binary restriction functor with an $(n-1)$-fold iterated induction product. To formulate this result clearly, we first introduce the corresponding index sets and operators.

For a sequence of dimension vectors $\gamma_1, \dots, \gamma_n$ and a pair of dimension vectors $\alpha, \beta$ such that $\gamma = \sum_{i=1}^n \gamma_i = \alpha + \beta$, we define the index set $\mathcal{N}_{n-1}$ as the collection of sequences of quadruples:
\begin{equation}
	\lambda = (\lambda_1, \lambda_2, \dots, \lambda_{n-1}) \in \prod_{j=1}^{n-1} (\mathbb{N}^{Q_0})^4.
\end{equation}
Each component $\lambda_j = \left( (\gamma_j)_a, (\gamma_j)_b, (\gamma_{\ge j+1})_a, (\gamma_{\ge j+1})_b \right)$ consists of the intermediate dimension vectors at the $j$-th step, which satisfy the following system of recursive constraints for $1 \le j \le n-1$:
\begin{equation}
	\begin{aligned}
		(\gamma_{\ge j})_a &= (\gamma_j)_a + (\gamma_{\ge j+1})_a, & (\gamma_{\ge j})_b &= (\gamma_j)_b + (\gamma_{\ge j+1})_b, \\
		\gamma_j &= (\gamma_j)_a + (\gamma_j)_b, & \gamma_{\ge j+1} &= (\gamma_{\ge j+1})_a + (\gamma_{\ge j+1})_b,
	\end{aligned}
\end{equation}
subject to the conditions $(\gamma_{\ge 1})_a = \alpha$ and $(\gamma_{\ge 1})_b = \beta$.

Associated with each sequence $\lambda \in \mathcal{N}_{n-1}$, the total shift is given by:
\begin{equation}
	\Omega = \sum_{j=1}^{n-1} \left( (\gamma_{n-j})_b, (\gamma_{\ge n-j+1})_a \right).
\end{equation}

Furthermore, we define a sequence of recursive operators $\Theta_j$ for $2 \le j \le n-1$ by:
\begin{equation}
	\begin{aligned}
		\Theta_j \coloneqq & \left(\mathrm{Id}_{\gamma_1,\dots, \gamma_{n-1-j}} \times \Ind_{(\gamma_{n-j})_a, (\gamma_{\ge n-j+1})_a}^{(\gamma_{\ge n-j})_a} \times \Ind_{(\gamma_{n-j})_b, (\gamma_{\ge n-j+1})_b}^{(\gamma_{\ge n-j})_b} \right) \circ \left( \mathrm{Id}_{\gamma_1,\dots, \gamma_{n-1-j}} \times \tau_{\lambda_{n-j}!} \right) \\
		& \circ \left( \mathrm{Id}_{\gamma_1,\dots, \gamma_{n-1-j}} \times \Res_{(\gamma_{n-j})_a, (\gamma_{n-j})_b}^{\gamma_{n-j}} \times \mathrm{Id}_{(\gamma_{\ge n-j+1})_a, (\gamma_{\ge n-j+1})_b} \right),
	\end{aligned}
\end{equation}
and the initial operator $\Theta_{\text{init}}$ by:
\begin{equation}
	\begin{aligned}
		\Theta_{\text{init}} \coloneqq & \left(\mathrm{Id}_{\gamma_1,\dots, \gamma_{n-2}}\times \Ind_{(\gamma_{n-1})_a, (\gamma_n)_a}^{(\gamma_{\ge n-1})_a} \times \Ind_{(\gamma_{n-1})_b, (\gamma_n)_b}^{(\gamma_{\ge n-1})_b} \right) \circ \left(\mathrm{Id}_{\gamma_1,\dots, \gamma_{n-2}} \times \tau_{\lambda_{n-1}!} \right) \\
		& \circ \left( \mathrm{Id}_{\gamma_1,\dots, \gamma_{n-2}}\times \Res_{(\gamma_{n-1})_a, (\gamma_{n-1})_b}^{\gamma_{n-1}} \times \Res_{(\gamma_{n})_a, (\gamma_{n})_b}^{\gamma_{n}} \right).
	\end{aligned}
\end{equation}

With these preparations, the main identity for the restriction of the iterated induction product can be stated as follows.

\begin{theorem} \label{thm:main2}
	Let $\mathcal{L}_i \in \mathcal{D}^b_{G_{\gamma_i}}(E_{\gamma_i}, \Ql)$ for $1 \le i \le n$ be semisimple complexes. Let $\gamma = \sum_{i=1}^n \gamma_i = \alpha + \beta$. The restriction of the iterated induction product satisfies the following isomorphism:
	\begin{equation}
		\Res_{\alpha,\beta}^{\gamma} \circ \bInd_{\gamma_1, \dots, \gamma_n}^{\gamma} (\boxtimes_{i=1}^n \mathcal{L}_i) 
		\simeq \bigoplus_{\lambda\in {\mathcal{N}_{n-1}}} \left( \Theta_{n-1} \circ \Theta_{n-2} \circ \dots \circ \Theta_{\text{init}} \right) (\boxtimes_{i=1}^n \mathcal{L}_i) \{-\Omega\}.
	\end{equation}
\end{theorem}

\section{Proof of the Iterated Formulas}

In this section, we provide the detailed inductive proof for Theorem \ref{thm:main1}. The proof for the dual identity (Theorem \ref{thm:main2}) follows an analogous logic and is thus omitted.

\begin{proof}[Proof of Theorem \ref{thm:main1}]
	We proceed by induction on $n \geq 2$.
	
	\medskip
	\noindent\textbf{1. Base Case ($n=2$):} 
	When $n=2$, the $(n-1)$-fold iterated restriction $\bRes_{\gamma_1, \gamma_2}^{\gamma}$ coincides with the binary restriction $\Res_{\gamma_1, \gamma_2}^{\gamma}$. According to the geometric Green's formula for semisimple complexes \cite{FLX}, we have the following isomorphism:
	\begin{equation}
		\Res_{\gamma_1, \gamma_2}^{\gamma} \circ \Ind_{\alpha, \beta}^{\gamma}(\mathcal{F} \boxtimes \mathcal{G}) 
		\simeq \bigoplus_{\lambda_1 \in \mathcal{M}_1} \Psi_{\text{final}} (\mathcal{F} \boxtimes \mathcal{G}) \{ -(\alpha_b, \beta_a) \}.
	\end{equation}
	By setting $n=2$ in the general formula, we observe that the operator chain consists of a single term $\Psi_{\text{final}}$, and the shift $\Xi$ reduces to the single symmetrized Euler form $( (\alpha_1)_b, (\beta_1)_a )$. Thus, the base case holds.
	
	\medskip
	\noindent\textbf{2. Inductive Hypothesis:} 
	Assume that the theorem holds for the case $n-1$. That is, for a sequence of dimension vectors $(\gamma_1, \dots, \gamma_{n-2}, \gamma_{\ge n-1})$ where $\gamma_{\ge n-1} = \gamma_{n-1} + \gamma_n$, the composition of the $(n-2)$-fold restriction with the induction functor yields:
	\begin{equation}
		\bRes_{\gamma_1, \dots, \gamma_{n-2}, \gamma_{\ge n-1}}^{\gamma} \circ \Ind_{\alpha, \beta}^{\gamma}(\mathcal{F} \boxtimes \mathcal{G}) 
		\simeq \bigoplus_{\lambda' \in \mathcal{M}_{n-2}} \left( \Psi'_{n-2} \circ \dots \circ \Psi_1 \right) (\mathcal{F} \boxtimes \mathcal{G}) \{ -\Xi_{n-2} \},
	\end{equation}
	where $\Psi'_{n-2}$ is the terminal operator for the sequence of length $n-1$, and $\Xi_{n-2} = \sum_{j=1}^{n-2} ( (\alpha_j)_b, (\beta_j)_a )$.
	
	\medskip
	\noindent\textbf{3. Inductive Step:} 
	Consider the $(n-1)$-fold iterated restriction $\bRes_{\gamma_1, \dots, \gamma_n}^{\gamma}$. That is, the case $n$. By the recursive definition of the iterated restriction functor, we can decompose the operator as:
	\begin{equation} \label{eq:proof-decomp}
		\bRes_{\gamma_1, \dots, \gamma_n}^{\gamma} = \left( \prod_{i=1}^{n-2} \mathrm{Id}_{\gamma_i} \times \Res_{\gamma_{n-1}, \gamma_n}^{\gamma_{\ge n-1}} \right) \circ \bRes_{\gamma_1, \dots, \gamma_{n-2}, \gamma_{\ge n-1}}^{\gamma}.
	\end{equation}
	Substituting the inductive hypothesis into \eqref{eq:proof-decomp}, we consider the action of the final restriction step on the output of the preceding $n-2$ operators. The operator $\Res_{\gamma_{n-1}, \gamma_n}^{\gamma_{\ge n-1}}$ slides past the first $n-2$ factors and interacts exclusively with the terminal induction product in $\Psi'_{n-2}$:
	\begin{align*}
		& \left( \prod_{i=1}^{n-2} \mathrm{Id}_{\gamma_i} \times \Res_{\gamma_{n-1}, \gamma_n}^{\gamma_{\ge n-1}} \right) \circ \left( \prod_{i=1}^{n-3} \mathrm{Id}_{\gamma_i} \times \Ind_{(\alpha_{n-2})_a, (\beta_{n-2})_a}^{\gamma_{n-2}} \times \Ind_{(\alpha_{n-2})_b, (\beta_{n-2})_b}^{\gamma_{\ge n-1}} \right) (\cdots) \\
		\simeq & \left( \prod_{i=1}^{n-3} \mathrm{Id}_{\gamma_i} \times \Ind_{(\alpha_{n-2})_a, (\beta_{n-2})_a}^{\gamma_{n-2}} \times \left( \Res_{\gamma_{n-1}, \gamma_n}^{\gamma_{\ge n-1}} \circ \Ind_{(\alpha_{n-2})_b, (\beta_{n-2})_b}^{\gamma_{\ge n-1}} \right) \right) (\cdots).
	\end{align*}
	We now apply the binary geometric Green's formula to the local composition $(\Res_{\gamma_{n-1}, \gamma_n}^{\gamma_{\ge n-1}} \circ \Ind_{(\alpha_{n-2})_b, (\beta_{n-2})_b}^{\gamma_{\ge n-1}})$. This yields a direct sum over $\lambda_{n-1}$ of the form:
	\begin{align*}
		& \left( \Ind_{(\alpha_{n-1})_a, (\beta_{n-1})_a}^{\gamma_{n-1}} \times \Ind_{(\alpha_{n-1})_b, (\beta_{n-1})_b}^{\gamma_n} \right) \circ \tau_{\lambda_{n-1}!} \\
		& \circ \left( \Res_{(\alpha_{n-1})_a, (\alpha_{n-1})_b}^{(\alpha_{n-2})_b} \times \Res_{(\beta_{n-1})_a, (\beta_{n-1})_b}^{(\beta_{n-2})_b} \right) \{ - ( (\alpha_{n-1})_b, (\beta_{n-1})_a ) \}.
	\end{align*}
	The terminal interaction is computed as:
	\begin{equation}
		\begin{aligned}
			&	\mathrm{Id}_{\gamma_1, \dots, \gamma_{n-3}} \times \mathrm{Ind}_{(\alpha_{n-2})_a, (\beta_{n-2})_a}^{\gamma_{n-2}} \times (\mathrm{Res}^{\gamma_{\ge n-1}}_{\gamma_{n-1}, \gamma_n} \circ \mathrm{Ind}_{(\alpha_{ n-2})_b, (\beta_{n-2})_b}^{\gamma_{\ge n-1}}) \\
			\simeq & \bigoplus_{\lambda_{n-1}} \mathrm{Id}_{\gamma_1, \dots, \gamma_{n-3}} \times \mathrm{Ind}_{(\alpha_{n-2})_a, (\beta_{n-2})_a}^{\gamma_{n-2}} \\ & \times
			\left[ \left( \mathrm{Ind}_{(\alpha_{n-1})_a, (\beta_{n-1})_a}^{\gamma_{n-1}}
			\times \mathrm{Ind}_{(\alpha_{n-1})_b, (\beta_{n-1})_b}^{\gamma_n} \right) \circ \tau_{\lambda_{n-1}!} 
			\circ \left( \mathrm{Res}_{(\alpha_{n-1})_a, (\alpha_{n-1})_b}^{(\alpha_{n-2})_b} \times \mathrm{Res}_{(\beta_{n-1})_a, (\beta_{n-1})_b}^{(\beta_{n-2})_b} \right)\right] \\
			\simeq & \bigoplus_{\lambda_{n-1}} \left(\mathrm{Id} \times \mathrm{Ind}_{(\alpha_{n-1})_a, (\beta_{n-1})_a}^{\gamma_{n-1}}
			\times \mathrm{Ind}_{(\alpha_{n-1})_b, (\beta_{n-1})_b}^{\gamma_n}\right) \circ \left(\mathrm{Id} \times \tau_{\lambda_{n-1}!}  \right) \\
			& \circ \left( \mathrm{Id} \times \mathrm{Res}_{(\alpha_{n-1})_a, (\alpha_{n-1})_b}^{(\alpha_{n-2})_b} \times \mathrm{Res}_{(\beta_{n-1})_a, (\beta_{n-1})_b}^{(\beta_{n-2})_b}\right) \circ \left( \mathrm{Id} \times  \mathrm{Ind}_{(\alpha_{n-2})_a, (\beta_{n-2})_a}^{\gamma_{n-2}} \times \mathrm{Id} \right) \\
			\simeq & \bigoplus_{\lambda_{n-1}} \Psi_{\text{final}} \circ \left( \mathrm{Id} \times  \mathrm{Ind}_{(\alpha_{n-2})_a, (\beta_{n-2})_a}^{\gamma_{n-2}} \times \mathrm{Id}
			\right).
		\end{aligned}
	\end{equation}
	By inserting this back into the inductive expression and re-indexing the summation over $\mathcal{M}_{n-1} = \mathcal{M}_{n-2} \times \{ \lambda_{n-1} \}$, we obtain the chain:
	\begin{equation}
		\Psi_{\text{final}} \circ \Psi_{n-2} \circ \dots \circ \Psi_1,
	\end{equation}
	where $\Psi_{n-2}$ is now a recursive (non-terminal) operator as defined in Theorem \ref{thm:main1}. 
	
	\medskip
	\noindent\textbf{4. The total shift:} 
	The total shift $\Xi$ is the sum of the total shift from the first $n-2$ steps and the local shift produced in the final step:
	\begin{equation}
		\Xi = \Xi_{n-2} + \left( (\alpha_{n-1})_b, (\beta_{n-1})_a \right) = \sum_{j=1}^{n-1} \left( (\alpha_j)_b, (\beta_j)_a \right).
	\end{equation}
	By Lemma \ref{lem: semisimplicity}, all intermediate complexes remain semisimple, ensuring that these local isomorphisms can be inductively composed. This completes the proof of Theorem \ref{thm:main1}.
\end{proof}

\begin{remark} \label{rem:complexity}
	Note that in the subcategory of semisimple complexes, the induction and restriction functors do not admit explicit decompositions. As a result, each recursive operator $\Psi_j$ must retain the intertwined structure of $\Ind \circ \tau \circ \Res$. The final formula thus appears as a nested alternating chain, which simply reflects the inherent complexity of the iteration in this setting.
\end{remark}

\end{document}